\documentclass[11pt,reqno]{amsart}
\usepackage[T1]{fontenc}
\usepackage{lmodern}
\usepackage[expansion=false]{microtype}
\usepackage{mathtools}
\usepackage{amssymb}
\usepackage{url}
\usepackage[hidelinks]{hyperref}
\newtheorem{theorem}{Theorem}[section]
\newtheorem{lemma}[theorem]{Lemma}
\newtheorem{corollary}[theorem]{Corollary}

\numberwithin{equation}{section}
\newcommand*{\doi}[1]{\href{https://doi.org/#1}{\nolinkurl{#1}}}
\title[A counterexample to Wormald's conjecture]{A counterexample to Wormald's conjecture}
\author{James Alexander Schreib}\thanks{The author's ORCID is \texorpdfstring{\url{https://orcid.org/0009-0005-1048-0209}}{0009-0005-1048-0209}.}
\address{New York University}
\email{jas10320@nyu.edu}
\date{6 September 2026}
\subjclass[2020]{05C70, 05C15}
\keywords{cubic graphs, linear forests, isomorphic partitions, formal verification}
\begin{document}
\begin{abstract}
We give a counterexample to Wormald's conjecture.
A 16-vertex cubic graph, the disjoint union of $K_{3,3}$ with a 10-vertex bridged graph built from two subdivided copies of $K_{4}$, has no partition of its edges into two isomorphic spanning linear forests.
The obstruction is parity: any such partition would force every monochromatic component to have even order while requiring a five-vertex side of a bridge to be partitioned into such components.
Adjoining copies of $K_{4}$ gives counterexamples in every order $16+4t$.
The counterexample is disconnected and leaves the connected case unresolved.
The full family is verified in Lean without custom axioms.
\end{abstract}
\maketitle
\section{Introduction}
\label{sec:intro}
The linear arboricity of a graph is the least number of linear forests needed to partition its edges.
Akiyama, Exoo, and Harary conjectured that this number is at most $\lceil (\Delta+1)/2 \rceil$ for maximum degree $\Delta$ \cite{akiyama1980}.
For cubic graphs this predicts a partition into two linear forests, proved by Akiyama, Exoo, and Harary \cite[Theorem~6]{akiyama1980}.
Wormald proposed a sharper form for orders divisible by four: the two forests can be taken spanning and isomorphic.
The conjecture appeared in 1987 as Problem~13 in Ars Combinatoria \cite{wormald1987}; see also Fouquet, Thuillier, Vanherpe, and Wojda \cite{fouquet2009} and Abreu, Goedgebeur, Labbate, and Mazzuoccolo \cite{abreu2019}.
Recent work of Kronenberg, Letzter, Pokrovskiy, and Yepremyan establishes the conjecture for all large connected cubic graphs \cite{kronenberg2022}.
The computations underlying Abreu et al.'s Corollary~4.6 and Observation~5.5 \cite{abreu2019} used snarkhunter, which generates connected cubic graphs \cite{brinkmann2011}; the disconnected counterexample below therefore lies outside the scope of that enumeration.
\begin{theorem}
\label{thm:main-intro}
There are finite simple cubic graphs of order divisible by four whose edges admit no partition into two isomorphic spanning linear forests.
\end{theorem}
The witness is a graph $G$ on 16 vertices described in Section~\ref{sec:graph}.
Its analysis uses only degree counts, one complement computation on $K_{3,3}$, and a parity argument across a bridge.
Section~\ref{sec:family} extends $G$ to counterexamples of every order $16+4t$.
Section~\ref{sec:formal} records the complete Lean verification and its scope.
\section{Preliminaries}
\label{sec:prelim}
Terminology follows Diestel \cite{diestel2017}.
Graphs are finite and simple.
Write $P_{j}$ for the path on $j$ vertices, so $P_{1}$ is isolated and $P_{2}$ is a single edge.
A linear forest is a disjoint union of paths.
It is spanning in a graph $G$ if it contains every vertex of $G$, and a partition of $E(G)$ into two spanning linear forests is identified with a red-blue edge coloring in which each color class is a linear forest.
Two linear forests are isomorphic if and only if their multisets of component orders agree.
A bridge is an edge whose deletion increases the number of components, and subdividing an edge replaces it by a path of length two through a new vertex.
\section{The graph}
\label{sec:graph}
Take two disjoint copies of $K_{4}$.
Subdivide one edge of each copy by a new vertex, $x$ in the first copy and $y$ in the second, and add $e=xy$.
Call the result $H$.
Each subdivided $K_{4}$ contributes five vertices, so $H$ has ten vertices; it is cubic, and $e$ is a bridge whose deletion leaves two components of five vertices each.
Let $G$ be the disjoint union of $K_{3,3}$ and $H$.
Then $G$ is cubic on $6+10=16$ vertices with $9+15=24$ edges, and it contains a bridge.
The claim is that $G$ admits no partition into two isomorphic spanning linear forests.
\section{Degree and \texorpdfstring{$K_{3,3}$}{K3,3} analysis}
\label{sec:degree}
Suppose the edges of $G$ are colored red and blue so that each color is a spanning linear forest, and suppose the two forests are isomorphic.
\begin{lemma}
\label{lem:degree}
At each vertex, the two color degrees are $1$ and $2$ in some order.
In particular neither forest has an isolated vertex.
\end{lemma}
\begin{proof}
The color degrees sum to $3$ and each is at most $2$, since a vertex of degree $3$ in a linear forest would contradict the path structure.
The only such pairs are $1+2$ and $2+1$.
\end{proof}
Since the forests are isomorphic and $G$ has 24 edges, each forest has 12 edges.
A forest with 12 edges spanning 16 vertices has $16-12=4$ components.
On the $K_{3,3}$ summand, let the restricted forests use $r$ and $b$ edges with $c_{R}$ and $c_{B}$ components.
Counting components in each restriction gives $c_{R}=6-r$ and $c_{B}=6-b$, while $r+b=9$.
Hence $c_{R}+c_{B}=12-9=3$.
Both counts are positive by Lemma~\ref{lem:degree}, so they are $1$ and $2$ in some order.
Exchange the colors so that red has one component on $K_{3,3}$: a $P_{6}$.
Write it as
\[
a_{1}-b_{1}-a_{2}-b_{2}-a_{3}-b_{3}.
\]
Its blue complement in $K_{3,3}$ is the $P_{4}$ on $b_{2}-a_{1}-b_{3}-a_{2}$ together with the $P_{2}$ on $a_{3}-b_{1}$.
Thus the restricted pair on $K_{3,3}$ is $P_{6}$ against $P_{4}\sqcup P_{2}$.
\section{Global inventory and parity}
\label{sec:parity}
The blue forest on $H$ has two components, since blue has four components globally and two lie on $K_{3,3}$.
Global isomorphism requires a blue $P_{6}$ matching the red $P_{6}$ on $K_{3,3}$.
That $P_{6}$ must lie on $H$, and its other blue component on the remaining four vertices of $H$ is a $P_{4}$.
Hence the full blue component multiset is
\begin{equation}
\label{eq:inv}
P_{6}\sqcup P_{4}\sqcup P_{4}\sqcup P_{2},
\end{equation}
and the red multiset agrees with it by isomorphism.
Every monochromatic component therefore has even order.
Now consider the color that omits the bridge $e$.
Every component of that color on $H$ lies wholly in one of the two five-vertex sides of $H-e$.
Its components meeting a fixed side partition those five vertices into even positive parts, impossible since a sum of even positive integers is even while $5$ is odd.
This contradiction establishes Theorem~\ref{thm:main-intro} for $G$ by \eqref{eq:inv}.
\section{Arbitrarily large counterexamples}
\label{sec:family}
\begin{corollary}
\label{cor:family}
For every integer $t\ge 0$, the disjoint union of $G$ with $t$ copies of $K_{4}$ is a cubic graph of order $16+4t$ with no partition into two isomorphic spanning linear forests.
\end{corollary}
\begin{proof}
In a linear forest partition of a $K_{4}$ summand, the restricted component counts sum to $8-6=2$.
Each is at least one, so both restrictions are $P_{4}$.
Each added $K_{4}$ therefore contributes the same path to each color.
Deleting these equal contributions from any hypothetically equal global path inventories would give a forbidden isomorphic partition of $G$.
\end{proof}
\section{Formal verification}
\label{sec:formal}
The concrete graph and full counterexample family are formalized in Lean~4.30.0 with Mathlib v4.30.0 \cite{moura2021,mathlib}, without custom axioms.
The graph $G$ is defined in \texttt{WormaldGraph.\allowbreak lean}, with bridge $10$--$15$.
The theorem \texttt{WormaldCounterexample.\allowbreak concrete\_\allowbreak counterexample} proves unconditionally that this graph is cubic, has order $16$, and admits no partition into two isomorphic spanning linear forests.

The proof checks the component-order inventories of all $512$ edge subsets of $K_{3,3}$ using a kernel-checked certificate checker, proves the forest edge/component identity and graph-sum inventory formulas, and derives the all-even inventory contradicting the odd bridge side.
Generated finite data is checked inside Lean; no external program is trusted.
The stronger theorem \texttt{WormaldGraph.\allowbreak no\_\allowbreak equal\_\allowbreak order\_\allowbreak partition} excludes even equal component-order multisets.
Corollary~\ref{cor:family} is proved for every $t\ge0$ by \texttt{WormaldGraph.\allowbreak family\_\allowbreak counterexample}, cancelling the component inventory $\{4\}$ contributed to each color by each added $K_4$.

The development is available at \url{https://github.com/jamesschreib/wormald-conjecture}.
Running \texttt{lake build} replays the verification, including the guarded checks in \texttt{AxiomAudit.\allowbreak lean}: the concrete and family theorems depend only on the standard foundational axioms \texttt{propext}, \texttt{Classical.choice}, and \texttt{Quot.sound}.
\section{Tool and computational resource disclosure}
ChatGPT~6 Astra (OpenAI) was accessed in chat to generate the proof. ChatGPT~5.6 Sol was accessed through Codex to write portions of the Lean~4 development described in Section~\ref{sec:formal}.\footnote{On the emerging conventions for such disclosure, see the Leiden Declaration on Artificial Intelligence and Mathematics~\cite{leiden2026}.}
All statements and proofs were subsequently verified independently by the author. The formal development is checked by the Lean kernel, and the Lean statements have been confirmed by hand to faithfully restate the corresponding statements of this paper. The author alone is responsible for the correctness of the arguments and for the accuracy and completeness of the citations.

\end{document}